\documentclass[12pt]{amsart}
\usepackage{amssymb,amsmath,graphics,verbatim}
\usepackage{latexsym}
\usepackage{eucal}
\usepackage{a4wide}
\usepackage[margin = 2.7cm]{geometry}
\usepackage{color}
\usepackage[utf8]{inputenc}
\usepackage[T1]{fontenc}

\newtheorem{theorem}{Theorem}[section]
\newtheorem{lemma}[theorem]{Lemma}

\newtheorem{prop}[theorem]{Proposition}
\newtheorem{cor}[theorem]{Corollary}

\def \N {\mathbb{N}}
\def \R{\mathbb{R}}
\def \C{\mathbb{C}}

\def \supp{\rm supp}

\def \Om{\Omega}

\def \Op{\mathrm{Op}}

\def \t{\tilde}
\def \half{\frac{1}{2}}

\numberwithin{equation}{section}

\begin{document}
\title{Partial Data Calder\'on Problem with Sharp Regularity on Admissible Manifolds}


\author[Tzou]{Leo Tzou}
\address{School of Mathematics and Statistics, University of Sydney, Sydney, Australia}
\email{leo@maths.usyd.edu.au}


\begin{abstract}
We solve the partial data Calder\'on problem on conformally transversallly anisotropic (CTA) manifolds with $L^{n/2}$ potentials - on par with sharp unique continuation result of Jerison-Kenig \cite{JerKen}. A trivial consequence of this is the sharp regularity improvement to the result of Kenig-Sj\"ostrand-Uhlmann \cite{ksu}. This is done by constructing a "Green's function" which possesses both desirable boundary conditions {\em and} satisfies semiclassical type estimates in the suitable $L^{p}$ spaces. No Carleman estimates were used in the writing of this article which makes it starkly different from the traditional approaches based on Bukhgeim-Uhlmann \cite{BukUhl} and Kenig-Sj\"ostrand-Uhlmann \cite{ksu}.

\end{abstract}

\maketitle

\begin{section}{Introduction}
The pioneering works of Bukgheim-Uhlmann \cite{BukUhl} and Kenig-Sj\"ostrand-Uhlmann \cite{ksu} on partial data Calder\'on problems for the Schr\"odinger operator $\Delta + q$ have inspired many works on the subject (see review article \cite{KenSalreview} and the references therein). Except cases where the domain geometry is trivial (e.g. flat/spherical boundary), all of them are based on the $L^2$ Carleman estimate approach developed by \cite{ksu}, \cite{BukUhl}, and \cite{KenSal}. 

In order to use the Carleman-based approach one must assume a-priori that $q\in L^\infty$ - an unsatisfactory assumption since we know that unique continuation holds even for potentials $q$ which are in $L^{n/2}$ (see \cite{JerKen}). Various {\em full} data Calder\'on problems in the $L^{n/2}$ limit were studied (see \cite{DosKenSal} and references therein) without the use of Carleman estimates. However, the techniques in \cite{DosKenSal} do not translate immediately to the more challenging partial data problems.

We propose a different method to solve partial data problems which bypasses the traditional Carleman approach. This allows us to minimize the assumption on the potential $q$ to $L^{n/2}$ - on par with the sharp assumption for unique continuations \cite{JerKen}.

Turns out that it is more convenient to apply this new approach in the more general geometric setting of "conformally transversally anisotropic" (CTA) manifold first introduced by \cite{dksu}. These are manifolds $M = \R\times M_0$ endowed with metrics conformal to $dy_1^2 \oplus g_0$ where $g_0$ is the metric on the closed compact manifold $M_0$. Suppose $\Omega \subset M$ is a smooth bounded domain is compactly contained in $\R \times \Omega_0$ where $\Omega_0\subset M_0$ is a simple domain. If $\Gamma_\pm\subset\partial\Omega$ are compactly contained in the sets 
$$\{y\in\partial \Omega \mid \pm g(\partial_{y_1}, \nu(y)) >0\}$$ 
where $\nu$ denotes the inward pointing normal, let $\Gamma_D := \partial M\backslash \Gamma_+$ and $\Gamma_N := \partial M \backslash \Gamma_-$. Let $q\in L^{n/2}(\Omega)$ and assume well-posedness of the Dirichlet BVP for $\Delta_g+ q$, denote by 
\[
\Lambda_q: H^{\half}(\partial \Om) \rightarrow H^{-\half}(\partial \Om) 
\]
the Dirichlet-Neuamm map. (See \cite{DosKenSal} for definition when $q\in L^{n/2}(\Omega)$.) 
We have the following theorem:

\begin{theorem}
\label{main theorem}
Let $q_1$ and $q_2$ be $L^{n/2}(\Omega)$ functions such that 
$$\Lambda_{q_1}f \mid_{\Gamma_N} = \Lambda_{q_2}f \mid_{\Gamma_N},\ \ \forall f\in C^\infty_0(\Gamma_D)$$ 
then $q_1 = q_2$.
\end{theorem}
Kenig-Salo \cite{KenSal} was the first to consider partial data type problems on CTA manifolds. The result for \cite{KenSal} is for sufficiently regular potentials whereas the focus of this article is on potentials which are unbounded. Partial data for unbounded potentials was also studied in \cite{chung-tzou} in the Euclidean setting for data on roughly half of the boundary. 

This result is new even in the Euclidean setting and can be seen as a sharp regularity (i.e. $q\in L^{n/2}$) version of the main theorem by Kenig-Sj\"ostrand-Uhlmann in \cite{ksu}. Indeed, consider a bounded smooth domain $\Omega\subset \R^n$, $z_0\in \R^n$ a point not in the closure of the convex hull of $\Omega$, and let $\Gamma_\pm \subset \partial \Omega$ be an open subsets compactly contained in 
$$\{z \in \partial \Omega \mid \pm \nu(z)\cdot (z-z_0) >0\}.$$
Define $\Gamma_D$ and $\Gamma_N$ as before, we have after a change of coordinates $y_1 = \log |z-z_0|$,
\begin{cor}
Let $q_1$ and $q_2$ be $L^{n/2}(\Omega)$ functions such that 
$$\Lambda_{q_1}f \mid_{\Gamma_N} = \Lambda_{q_2}f \mid_{\Gamma_N},\ \ \forall f\in C^\infty_0(\Gamma_D)$$ then $q_1 = q_2$.
\end{cor}
Observe that geometrizing the problem from the Euclidean case to the more general setting of Theorem \ref{main theorem} "linearizes" the $\log$ variable which allows us to adapt the parabolic flow construction of \cite{Chuthesis} to our setting. The price we pay, of course, is that the underlying geometry becomes more involved and the Fourier multiplier Green's function constructed by \cite{su} is no longer suitable. To remedy this difficulty we use instead the Green's function constructed via Fourier series. These were first used in \cite{KenSalUhl} and later in \cite{DosKenSal}.

We will take the Green's function of \cite{KenSalUhl} and transform it into a {\em Dirichlet} Green's function which will be the key to solving our problem. To simplify notations fix throughout this article $p := \frac{2n}{n-2}$ and $p' := \frac{2n}{n+2}$. We write $T:X \rightarrow_{h^m} Y$ to denote $\| T\|_{X\to Y} \leq C h^m$:
\begin{prop}
\label{green's function}
There exists a Green's function
\[G_{\Gamma_\pm}: L^2(\Omega)\to_{h^{-1}} L^2(\Omega),\ G_{\Gamma_\pm} : L^{p'}(\Omega)\to_{h^{-2}} L^p(\Omega).\] 
which resolves the conjugated Laplacian $h^2 e^{\mp y_1/h}\Delta_g e^{\pm y_1/h} G_{\Gamma_\pm} = Id$ in $\Omega$.
Furthermore, $G_{\Gamma_\pm} v\in H^1(\Omega)$ for all $v \in L^{p'}$  and $G_{\Gamma_\pm} v \mid_{{\Gamma_\pm}} = 0$.
\end{prop}

The theory and methods developed in this article along with earlier work in this direction \cite{chung-tzou} has a central theme: Given any Green's function for the conjugated Laplacian with suitable semiclassical $L^p$ estimates, there is a systematic way to "upgrade" it to one which obeys the Dirichlet boundary condition while simultaneously preserving the same estimates. One can, of course, adapt similar methods of this article and \cite{chung-tzou} to obtain different types of boundary conditions (e.g. Neumann, Robin, etc.) and different types of conjugation with elliptic operators (Dirac, bi-Laplace, etc). As these Dirichlet Greens functions are the pivotal piece in many inverse problems \cite{ksu, KruUhl32, Chuthesis}, unique continuations and Carleman estimates \cite{jerome1, jerome2, jerome3, jerome4, jerome5}, we anticipate that their scope of application extends beyond Calder\'on problems. The explicit nature of their construction, bypassing the traditional route of abstract functional analysis machinery (see \cite{ksu}), also gives hope for the possibility of CGO based numerical reconstruction algorithms with partial data in the spirit of \cite{Nac, num1, num3}.

In Section \ref{psido} we review the $\Psi$DO calculus which we will use. In Section 3 we construct solutions to a parabolic flow in phase space - this will provide us with the Dirichlet data on the desired portion of the boundary. In Section 4 we use this flow to construct the Green's function which is the key piece for proving Theorem \ref{main theorem}. In Section 5 we construct the CGO using these Green's functions and finally employ them in Section 6 to prove Theorem \ref{main theorem}.

\end{section}

\begin{section}{Elementary Semiclassical $\Psi$DO theory on CTA Manifolds}\label{psido}
We review some notions about semiclassical pseudodifferential operators which will be relevant to the rest of this article. Throughout this article we will use the Weyl quantization to produce operators acting on sections of the half-density bundle ${\bf \Omega}^{1/2}(M)$, which we identify with the trivial line bundle via the volume form.  This has the advantage that symbols of semiclassical operators in $\Psi^k(M)$ are defined up to $h^2 S^{k-2}(M)$. Proofs of the results in this section are contained are omitted as they involve application of standard techniques described in articles such as \cite{chung-tzou}, \cite{wong}, and \cite{zworski semiclassical}.  
\begin{subsection}{Semiclassical Sobolev Spaces}

We use semiclassical Sobolev spaces with the norm $\|u\|_{W^{k,r}(M)} := \| \langle hD\rangle^k  u\|_{L^r}$,
which, for $k\in \N$ is equivalent to the one involving derivatives $ \sum\limits_{|\alpha |\leq k} \| (hD)^\alpha u\|^r_{L^r}$.

Let $M = \R\times M_0$ where $M_0$ is a closed compact manifold with metric $g_0$ and consider the metric $dx_1^2 \oplus g_0$ which makes $M$ a transversally anisotropic manifold (\cite{dksu}). Denote the elements of $M$ by $(x_1, x')$. The cotangent bundle of $M$ has a natural splitting $T^*M = T^*\R \oplus T^*M_0$ whose elements we write as $(\xi_1,\xi')$.
We define the mixed Sobolev norms for $u \in C^{\infty}_c(M)$ by 
$$\|u\|_{W^{k,r}(M_0)W^{\ell, r}(M)} := \| \langle hD'\rangle^k \langle hD\rangle^\ell u\|_{L^r}$$
and the space $W^{k,r}(M_0)W^{\ell, r}(M)$ by completion. To simplify, we drop the $M_0$ and $M$ in the parentheses and use the notation that the first $W^{k,r}$ denotes action by 
$\langle hD'\rangle^k$ and the second $W^{\ell,r}$ denotes action by $\langle hD\rangle^\ell$. Note that if $k \geq 0$, 
\begin{eqnarray}\label{mixed space embedding}W^{-k,r} W^{\ell,r } \subset W^{\ell-k, r}(M).\end{eqnarray} 

In addition to H\"ormander symbols $S^{\ell}_1 (M)$, we will also consider symbols in the class $S^k_0(M)$ which do not decay when differentiated with respect to $\xi$:
\[|\partial^\alpha_{x}\partial^\beta_{\xi} a(x,\xi) |\leq C_{\alpha,\beta} |\xi|^k\]

We denote by the symbol space $S^k_1(M_0)S^\ell_j(M)$ by product symbols of the form $ba(x',\xi)$ where $b(x',\xi')\in S^k_{1}(M_0)$ and $a(x',\xi)\in S^\ell_j(M)$ for $j =0,1$. Again, to simplify notation we will drop the $M_0$ and $M$ and just write $S^k_1S^\ell_j$.

Boundedness of quantization of symbols $a\in S^0_1(M)$ are given by the Calder\'{o}n-Vaillancourt estimate (semiclassical version): for all $1<r<\infty$ and $h>0$ sufficiently small there exists a constant $k(n)$ such that
\begin{eqnarray}\label{semiclassical calderonvaillancourt}\|a(x,hD) u\|_{L^r} \leq \sum\limits_{|\alpha|,|\beta| \leq k(n)} p_{\alpha,\beta}(a) \|u\|_{L^r} + C\sqrt{h}\|u\|_{L^r}.\end{eqnarray}
where $p_{\alpha,\beta}(a) := \sup\limits_{x,\xi} |\partial_x^\alpha \partial_\xi^\beta a(x,\xi)| \langle \xi\rangle^{|\beta|}$. In what follows, if $n$ is the dimension let $k(n)$ be the smallest integer for which the estimate \eqref{semiclassical calderonvaillancourt} is true.

We have the following mapping properties for operators with symbols in $S^{k}_1 S^{-\ell}_1 \cup S^{k}_1 S^{-k(n) -\ell}_0$ (see \cite{chung-tzou}).
\begin{prop}
\label{sobolev mapping}
If $b(x',\xi')\in S^{k}_1$ and $a(x',\xi) \in S^{\ell}_1 \cup S^{-k(n)+\ell}_0$ then the quantization of the product symbol $ab(x'\xi)$ maps $W^{m,r} W^{l,r} \to_{h^0} W^{m-k,r} W^{l - \ell,r}$.

\end{prop}
Composition of these operators are given by
\begin{prop}
\label{sobolev composition}
If $a\in S^{k_1}_1S_1^{\ell_1} \cup S^{k_1}_1 S^{-k(n) +\ell_1}_0$ and $b\in S^{k_2}_1S_1^{\ell_2} \cup S^{k_2}_1 S^{-k(n) +\ell_2}_0$ then the composition $b(x'hD)a(x',hD)$ is given by
\[ab(x',hD) + \frac{h}{2i}\{a,b\}(x',hD) + h^2 m(x',hD)\]
with the remainder $m(x',hD)$ a bounded map from $W^{k,r}W^{\ell,r} \to W^{k - k_1-k_2,r}W^{\ell-\ell_1-\ell_2,r}$. 
\end{prop}

\end{subsection}
\end{section}

\begin{section}{Parabolic Equation}

Denote by $M_+ := (0,\infty)\times M_0 \subset M$ and $M_- := (-\infty,0)\times M_0\subset M$. Let $B(x',\xi') \in S^1_1(M_0)$, and define 
\begin{equation}\label{J}
j(x',hD) := h\partial_{x_1} + B(x',hD').
\end{equation}

Estimate \eqref{semiclassical calderonvaillancourt} gives that that $j(x',hD):W^{1,r}(M) \rightarrow L^r(M)$ for $1 < r < \infty$. In this section we follow \cite{chung-tzou} and derive some properties of its inverse.

We assume that the real part of $B(x',\xi')$ obeys the ellipticity condition
\begin{equation}\label{EllipticF}
c\langle \xi' \rangle  \leq \mathrm{Re}B(x',\xi') \leq C\langle \xi' \rangle
\end{equation}
uniformly in $x'$, for some constants $c, C > 0$ which ensures the ellipticity of $j(x,\xi) := i\xi_1 + B(x',\xi')$. 

Unfortunately even with ellipticity the symbol $(i\xi_1 + B)^{-1}$ is not in general in the class $S^{-1}_1(M)$ (but rather in $S^{-1}_0(M)$). We need to assume that there is a first order symbol $i\xi_1 + B_-(x',\xi')$ which is elliptic outside of a compact set, such that
\begin{eqnarray}
\label{complete the square}
(i\xi_1 + B) (i\xi_1+ B_-) = {\mathcal P}(x',\xi) + a_0
\end{eqnarray}
for some second order polynomial (in $\xi)$) ${\mathcal P}(x',\xi)$ which is elliptic outside a compact set and $a_0 \in S^{-\infty}(M_0)$. It was shown in \cite{chung-tzou} that \eqref{complete the square} implies 
\begin{equation}\label{expansion of inverse}
j^{-1}(x',\xi) = (i\xi_1 + B)^{-1} \in  S^0_1 S^{-1}_1 + S^{-\infty} S_{0}^{-1-k(n)} + S_1^1 S^{-2}_1 
\end{equation}
which then implies
\begin{eqnarray}
\label{j parametrix}
 j^{-1}(x',hD) : L^r \to W^{1,r},\ \ 1<r<\infty.
\end{eqnarray}

One useful property is that $j^{-1}(x',hD)$ is equivalent to solving the Cauchy problem for the parabolic flow with initial condition on $x_1 = 0$. Indeed, let $U\subset M_0$ be a coordinate chart and $u$ be a smooth function which is compactly supported in the (infinite) strip $\R \times U$. Identifying $u$ with its pull-back by the coordinate map we can write 
\[
j^{-1}(x',hD) u(x_1,x') = h^{-n}\int_{\R^{n-1}}\int_{\R^{n-1}}e^{\frac{i}{h}(x'-y')\cdot \xi'} \int_{\R}\int_{\R}\frac{e^{\frac{i}{h}(x_1-s)t}u(s,y')}{it + B(\frac{x'+y'}{2},\xi')}  dt ds\, d\xi'dy'.
\]  
The inner integral can be computed using residue theorem to obtain
\[j^{-1}(x',hD)u(x_1,x')  = h^{-n}\int_{\R^{n-1}}\int_{\R^{n-1}}e^{\frac{i}{h}(x'-y')\cdot \xi'}  \int_{-\infty}^{x_1} e^{\frac{s-x_1}{h} B(\frac{x'+y'}{2},\xi')} ds dy'd\xi'.\]
Therefore, $j^{-1}(x',hD) u(x_1,x')\mid_\{x_1 \leq 0\} =0$. A partition of unity argument on the compact manifold $M_0$ shows that this holds for all $u \in C_c^\infty(M_+)$. A density argument allows us to conclude that
\begin{eqnarray}
\label{support of j^-1}
j^{-1}(x',hD)u \in W^{1,r}(M),\ \ j^{-1}(x',hD)u \mid_{x_1 \leq 0} = 0
\end{eqnarray}
if $u\in L^r(M)$ and $u\mid_{M_-} =0$.

Henceforth we will refer to the support property given by \eqref{support of j^-1} as ``preserving support in $M_+$".

Standard semiclassical calculus allows us to turn $j^{-1}(x',hD)$ into an inverse of $j(x',hD)$. First observe that such if $a(x',\xi') \in S^{1}_1( M_0)$ \eqref{expansion of inverse} and Proposition \ref{sobolev composition} yields
{\small\begin{eqnarray}
\label{composing with j^-1}
a(x',hD') j^{-1}(x',hD) = (aj^{-1})(x',hD) +\frac{ h}{2i} (j^{-2} \{a, B\})(x',hD)+ h^2m(x',hD)
\end{eqnarray} }
where $m(x',hD)$ and $ (j^{-2} \{a, B\})(x',hD)$ map $L^{r} \to_{h^0} L^r$.
Using this composition formula and Proposition \ref{sobolev composition} we invert $J:= j(x',hD) + hB_0(x',hD')$:
\begin{prop} 
\label{J inverse}
Let $J := j(x',hD) + hB_0(x',hD)$ for some $B_0(x',\xi') \in S^{0}_1(M_0)$. For $h>0$ sufficiently small there exists $J^{-1}: L^r \to W^{1,r}$ of the form
\[J^{-1}=j^{-1}(x', hD)(1+  h (j^{-1}B_0)(x',hD) + h^2 m_2(x',hD)) \]
and preserving $M_+$ support. 
Here the correction term $m_2$ satisfies $m_2(x',hD): L^r \rightarrow_{h^0} L^r$. \end{prop}

One final consequence of the structure of $J^{-1}$ is the following disjoint support property of \cite{chung-tzou}. We sketch the proof for the convenience of the reader:
\begin{lemma}
\label{disjoint support for J inverse}
Let ${\bf 1}_{M_-} \in L^\infty(M)$ be the indicator function for $M_-$ and let $\epsilon >0$ be small. Then for all $f\in L^r(M)$,
\[
\|J^{-1}{\bf 1}_{M_-}f \|_{W^{1,r}(\{x_1 \geq \epsilon\})} \leq C_\epsilon h^2 \|f\|_{L^r}.
\]
\end{lemma}

\begin{proof}
Let $\zeta_{\epsilon}(x_1)$ be a smooth function of one variable such that $\zeta_\epsilon(x_1) = 1$ on $\{x_1 \geq \epsilon \}$ and $\zeta_\epsilon(x_1) = 0$ on an open set containing $\overline M_-$. Then 
\[
\|J^{-1}{\bf 1}_{M_-}f \|_{W^{1,r}(\{x_1 \geq \epsilon\})} \leq \|\zeta_{\epsilon}J^{-1}{\bf 1}_{M_-}f \|_{W^{1,r}(M)}.
\]
Therefore it suffices to show that 
\begin{eqnarray}
\label{gap estimate}
\|\zeta_{\epsilon}J^{-1}{\bf 1}_{M_-}f \|_{W^{1,r}(M)} \leq C_\epsilon h^2 \|f\|_{L^r}.
\end{eqnarray}

From Proposition \ref{J inverse}, we have that 
\[J^{-1}=j^{-1}(x', hD)(1+  h (j^{-1}B_0)(x',hD) + h^2 m_2(x',hD)) \]
We will only show \eqref{gap estimate} for the principal part $\zeta_\epsilon j^{-1}(x',hD) {\bf 1}_{M_-}$ since the lower order terms can be handled using the terms in \ref{J inverse}. Writing $j^{-1}(x',\xi)$ using expansion \eqref{expansion of inverse} we see that the desired estimate is a special case of disjoint support property for operators of the type $\zeta_\epsilon ba(x',hD) {\bf 1}_{M_-} $ for symbols $a$ and $b$ in the suitable symbol class.

\end{proof}

\end{section}

\begin{section}{Dirichlet Green's Function}
In this section we assume that the metric $g$ on $\R \times M_0$ takes the form $dy_1^2 \oplus g_0$ (i.e. no conformal factor). Let $\Omega \subset M = \R\times M_0$ be a smooth bounded domain contained in $I\times M_0$ for some compact interval $I\subset \R$. If $\Gamma_\pm\subset \partial \Omega$ is open and compactly contained in $\{ y\in \partial \Omega \mid \pm g(\partial_{y_1}, \nu(y)) >0\}$, we would like to invert 
\[h^2\Delta_\pm := h^2 e^{\mp y_1/h}\Delta e^{\pm y_1/h},\]
with Dirichlet boundary condition on $\Gamma_\pm$ and have the inverse satisfy good $L^{p'}\to L^p$ estimates. Note that in this geometric setting every connected component of $\Gamma_\pm$ can be expressed as a portion of the graph of a smooth function $\{y_1 = f(y')\}$ with $f\in C^\infty(M_0)$. For the purpose of simplifying notation we will only work with the "$+$" sign and set $\Gamma:=\Gamma_+$. The theory we develop here works equally well for the "$-$" sign.

We begin with the result of \cite{DosKenSal} and \cite{KenSalUhl}. Let $ g= dy_1^2 \oplus g_0$ be the metric on the product manifold $M = \R \times M_0$. Kenig-Salo-Uhlmann in \cite{KenSalUhl} constructed a Green's function $G^{M}_\pm$ solving $e^{\mp y_1/h} h^2\Delta_g e^{\pm y_1/h} G^{M}_\pm = Id$ of the form
\[G^{M}_\pm f = \sum \limits_{l} \int_{-\infty}^\infty \frac{ e^{iy_1 \xi_1}\hat f_l(\xi_1,y')}{ h^2\xi_1^2 \mp 2ih\xi_1 -1 + h^2\lambda_l} d\xi_1\]
where $f_l(y_1,y') := e_l(y')\int_{M_0} e_l f(y_1,\cdot) dy'$, $\Delta_{g_0} e_l = \lambda_l e_l$, and $\hat f_l$ denotes the Fourier transform in the $y_1$ direction. By \cite{DosKenSal} for any compactly supported functions $\chi,\tilde\chi\in C^\infty(\R)$ one has 
\begin{eqnarray}
\label{DosKenSal estimate}
\t\chi(y_1)G^M_\pm \chi(\cdot): L^2 \to_{h^{-1}} H^2,\ \ \ \t\chi(y_1)G^M_\pm\chi(\cdot): L^{p'} \to_{h^{-2}} L^{p}.
\end{eqnarray}

This operator does not satisfy the desired boundary conditions along $\Gamma$ so more work will be needed. To this end we need to derive some finer properties for this operator. In particular, we would like to show that away from the characteristic set of $e^{\mp y_1/h} h^2\Delta_g e^{\pm y_1/h}$ the operator $G^{M}_\pm$ behaves more or less like a $\Psi$DO:
\begin{lemma}
\label{cut up G}
Let $\rho(y',\xi)\in S^{-\infty}$ be any symbol such that $\rho = 1$ in a neighbourhood of the set $\{\xi_1 = 0, |\xi'|_{g(y)} = 1\}$ then 
\begin{eqnarray*}
G^{M}_\pm &=& \Op(S_1^{-2}) + \Op(S_1^0) \rho(hD) G^{M}_\pm =  \Op(S_1^{-2}) +  G^{M}_\pm \rho(hD)\Op(S_1^0)\\
&=&  \Op(S_1^{-2})  +  \Op(S_1^0) \rho(hD) G^{M}_\pm   \rho(hD)\Op(S_1^0).
\end{eqnarray*}
Furthermore, one can write
\[\t\chi(y_1) G_\pm^M\chi(\cdot)  = (\t\chi(y_1) G_\pm^M\chi(\cdot))^c + \Op(S^{-2}_1) \]
where 
$$ (\t\chi(y_1) G_\pm^M\chi(\cdot))^c : L^2 \to_{h^{-1}} H^k,\  (\t \chi(y_1) G_\pm^M\chi(\cdot))^c : L^{p'} \to_{h^{-2}} W^{k,p}$$

\end{lemma}
\begin{proof}
The first statement comes from ellipticity of $e^{\mp y_1/h} h^2\Delta_g e^{\pm y_1/h}$ away from the support of the symbol $\rho$ and one can construct both left and right semi-classical parametrix.

For the last statement, choose $\rho,\tilde \rho \in S^{-\infty}(M)$ compactly supported on each fiber such that $\rho = 1$ is supported in a compact neighbourhood of the characteristic set and $\t \rho = 1$ on $\supp(\rho)$. We can write 
\[\t\chi G^M_\pm \chi = \t\rho(hD)  \t\chi G^M_\pm \chi + (1-\t \rho(hD))\t\chi G^M_\pm \chi.\]
For the first term, setting $(\t\chi G^M_\pm t\chi )^c :=  \t\rho(hD)  \t\chi G^M_\pm t\chi$ we have by \eqref{DosKenSal estimate}
$$ (\chi(y_1) G_\pm^M\chi(\cdot))^c : L^2 \to_{h^{-1}} H^k,\  (\chi(y_1) G_\pm^M\chi(\cdot))^c : L^{p'} \to_{h^{-2}} W^{k,p}.$$
Substituting $G^M_\pm =   \Op(S_1^{-2})  +  \Op(S_1^0) \rho(hD) G^{M}_\pm   \rho(hD)\Op(S_1^0)$ into the second term completes the proof.
\end{proof}
Another useful statement about the microlocal support of the Green's function is the following 
\begin{lemma}
\label{microlocal wf}
If the support of $a \in S^{1}_0(M)$ is disjoint from the characteristic set $\{\xi_1 = 0, |\xi'| = 1\}$ then 
\[ \|a(y,hD) \t\chi(y_1) G_\pm^M\chi(\cdot) \|_{ L^2\to L^2} \leq C,\ \ \|a(y,hD) \t\chi(y_1) G_\pm^M\chi(\cdot) \|_{ L^{p'}\to L^2} \leq Ch^{-1}\]
 for any compactly supported function $\chi,\tilde\chi \in C^\infty_0(\R)$.

\end{lemma}
\begin{proof}
In the proof of Lemma \ref{cut up G} choose $\rho,\t\rho \in S^{-\infty}(M)$ so that their supports are disjoint from $a(y,\xi)$. 
The second statement comes directly from using the first statement to write
\begin{eqnarray*}
a(y,hD) \t\rho(y',hD) \tilde \chi G^M_{\pm} \chi = a(y,hD) \t\rho(y',hD) \tilde \chi  \Big(\Op(S_1^{-2}) +  G^{\R\times M_0}_\pm\Op(S_1^0(M)) \rho(y', hD)\Big) \chi.
\end{eqnarray*}
This clearly maps $H^{-k}(M) \to H^k(M)$ for all $k$ with decay $h^\infty$. Therefore it suffices to analyze the mapping properties of $a(x,hD) (1-\t\rho(x',hD)) \tilde \chi G^M_{\pm} \chi$. Substituting 
$$G^M_\pm =   \Op(S_1^{-2})  +  \Op(S_1^0) \rho(hD) G^{M}_\pm   \rho(hD)\Op(S_1^0)$$ into this term completes the proof.
\end{proof}

In order to deal with the fact that portions of $\partial \Omega$ which are described by graphs of smooth functions $y_1 = f(y')$, we will consider portions of the boundary which are straight in the coordinate given by
\begin{eqnarray}
\label{change of var}
\gamma : (y_1,y') \mapsto (x_1 , x') = (y_1 - f(y'), y').
\end{eqnarray}
In these coordinates, the push-forward of the conjugated Laplacian $h^2\t\Delta_+:= \gamma_*(e^{- y_1/h}h^2 \Delta_g e^{ y_1/h})$ is
{\Small \begin{eqnarray}
\label{quantize the laplacian}
h^2 \tilde\Delta_+ 
= (1+ |df|^2) \Op\Big(\xi_1^2 - 2\xi_1\frac{ i-g_0(\xi',df)+ hi F}{1 + |df|^2} - \frac{1- |\xi'|^2+hi \xi'(K)}{1+ |df|^2}\Big) + h^2 \Op(S^0_1(M_0))
\end{eqnarray}}
for some real valued $F\in C^\infty(M_0)$ and $K\in C^\infty(M_0; TM_0)$. Observe that $\t G_+^M := \gamma_* G^M_+ $ is a Green's function for $h^2\t\Delta_+$. If $I\subset \R$ is a compact interval, choose $\tilde\chi,\chi\in C^\infty_0(M)$ which is equal to $1$ on $I\times M_0$ with $\tilde\chi \chi = \chi$. We can also require that $(\gamma_*\t \chi)(x_1,x')$ is a function of $x_1$ only so that 
$$[\Op(S^k_0(M_0),(\gamma_* \t\chi) ] = [\Op(S^k_0(M_0),(\gamma_* \chi) ] = 0.$$
Define
\begin{eqnarray}
\label{green on M}
G^I_+ := \t\chi G^M_+ \chi,\ \ \ \t G^I_+ := \gamma_* G^I_+
\end{eqnarray} 
As a consequence of Lemma \ref{cut up G} we have that 
\begin{eqnarray}
\label{compact green estimate}
\t G^I_+ = (\t G^I_+)^c + \Op(S^{-2}_1),\ \ G^I_+ = ( G^I_+)^c + \Op(S^{-2}_1) 
\end{eqnarray}
where 
$$ (G^I_+)^c, (\t G^I_+)^c: L^2 \to_{h^{-1}} H^k,\  (G^I_+)^c, (\t G^I_+)^c: L^{p'} \to_{h^{-2}} W^{k,p}.$$

\begin{subsection}{Decomposition of $\t\Delta_+$}
It was observed in \cite{Chuthesis} that the principal symbol of $\frac{1}{1+|df|^2}h^2\tilde \Delta_+$ factors formally as
{\tiny\begin{eqnarray*}\xi_1^2 -2\xi_1\frac{(i-g_0(df,\xi')-i hF)}{ 1+ |df|^2} -\frac{(1- |\xi'|^2 -ih\xi'(K))}{1+|df|^2}&=& \left(\xi_1 - i\left( \frac{(1 + ig_0(\xi',df) + hF)-r_0}{1+|df|^2}\right)\right)\times\\&&\left(\xi_1 - i\left( \frac{(1 + ig_0( \xi',df)+hF)+r_0}{1+|df|^2}\right)\right)
\end{eqnarray*}} 
where 
\begin{eqnarray}
\label{r_0}
r_0 := \Big({(1 + i g_0(\xi',df) +hF)^2 - (1-|\xi'|^2 - ih \xi'(K))(1+ |df|^2)}\Big)^{1/2}
\end{eqnarray}
is the standard branch of the square root. To avoid the discontinuity of the square root we make the following modification:\\ We see from the argument of the square-root that the (standard) branch cut occurs on the set
{\small $$\{2(1+hF)g_0(df,\xi') + h(1+|df|^2)\xi'(K) = 0\}\cap \{
(1+hF)^2 \leq (1-|\xi'|^2)(1+|df|^2) + h^2 \big( \frac{(1+|df|^2) \xi'(K)}{2(1+hF)}\big)^2\}.$$}

To avoid this set, observe that for all $\delta >0$ there exists an $\epsilon >0$ and $h_0>0$ such that if $2(1+hF)g_0(df,\xi') + h(1+|df|^2)\xi'(K) = 0$ and $|\xi'|^2 \geq \sup\limits_{x\in M_0}\frac{|df|^2}{1+|df|^2}+\delta$ then
\[ (1+hF)^2 \geq \epsilon + (1-|\xi'|^2)(1+|df|^2) + h^2 \big( \frac{(1+|df|^2) \xi'(K)}{2(1+hF)}\big)^2,\ \ \forall h< h_0\]
Therefore we choose constants $0<c< c'<1$ such that $\frac{|df|^2}{1+|df|^2} <c$ uniformly and let $\tilde\rho_0(\xi')$ be compactly supported such that $\tilde \rho_0 = 1$ whenever $|\xi'|^2 \leq c$ and its support is compactly contained in   $B_{\sqrt{c'}}$. Introduce another $\tilde \rho$ such that $\tilde \rho = 1$ on $|\xi'|^2 \leq c'$ but whose support is compactly contained in $B_{1}$. Observe that
\begin{eqnarray}
\label{supp of rho tilde}
\left|\xi_1^2 -2\xi_1\frac{(i-g_0(df,\xi') - hF)}{ 1+ |df|^2} -\frac{(1- |\xi'|^2 - \xi'(K))}{1+|df|^2}\right| >0.
\end{eqnarray}
uniformly over the set $\{\xi\in {\supp}{\t\rho},\ x'\in M_0\}$.
The discontinuity of the square root occurs within $|\xi'|^2 \leq |df|^2(1 + |df|^2)^{-1}$. This means that for $\xi'$ in $\supp(1-\tilde\rho_0)$, the function 
$$(1 + i g_0(\xi',df) +hF)^2 - (1-|\xi'|^2 - ih \xi'(K))(1+ |df|^2)$$ 
stays uniformly away from the discontinuity. This means that
\begin{eqnarray}
\label{square root}
r := (1- \t\rho_0) r_0
\end{eqnarray}
is a smooth symbol. We can now decompose $\xi_1^2 -2\xi_1\frac{(i-g_0(df,\xi')-i hF)}{ 1+ |df|^2} -\frac{(1- |\xi'|^2 -ih\xi'(K))}{1+|df|^2}$ as
{\small \begin{eqnarray}
\label{decomposition}
{(\xi_1 - \t a_- +h m_0) }{(\xi_1 - \t a_+ -h m_0)} + \t a_0 + h\{\t a_- ,\t a_+\} - h{m_0 \t a_-}+ h^2 m_0^2
\end{eqnarray}}
with $m_0(x',\xi') := - \t a_+^{-1}\{\t a_- ,\t a_+\}$. Here the $\t a_\pm \in S^1_1(M_0)$ and $\t a_0 \in S^{-\infty}(M_0)$ are defined by 
{\Small\begin{eqnarray}
\label{the a's}
\t a_0(x',\xi') &:=& \t\rho_0(-2+\rho_0)\Big((1+ig_0(df,\xi') + hF )^2 -(1-|\xi'|^2 - ih\xi'(K))(1+ |df|^2)\Big)\\\nonumber
\t a_\pm(x',\xi') &:=& i\frac{(1+ig_0(df,\xi') + hF ) \pm r}{2}.
\end{eqnarray}}
We also denote by $\t A_0$ and $\t A_\pm$ their respective quantizations. Observe that $\supp(\t a_0)$ is uniformly bounded away from the support of $\supp(1-\tilde \rho)$.
Quantizing the factorization \eqref{decomposition} we see that
\begin{eqnarray}
\label{operator factorization}
\frac{1}{1+|df|^2} h^2\tilde\Delta_+ 
= QJ + \t A_0 - h\tilde E_1 +  h^2 \tilde E_0 + h^2 \Op(S^0_1(M_0)).
\end{eqnarray}
Here $\tilde e_1 = m_0 \t a_- \in S^{1}_1(M_0)$, $\tilde e_0 \in S^0_1(M_0)$, and $Q$ and $J$ are associated to the symbols $\xi_n - \tilde a_- + hm_0$ and $\xi_n - \tilde a_+ + hm_0$ respectively. Again, here we use capitalization to denote the quantization of a symbol. We have the following estimate:

\begin{lemma} 
\label{composition estimate}
If $\t G_+^I$ is the Green's function defined by \eqref{green on M}, then there is an operator $(\tilde E_1 \tilde G_+^I)^c$ with bounds
\[(\t E_1 \t G_+^I)^c: L^2 \to_{h^0} H^k,\ \  (\t E_1 \t G_+^I)^c:L^{p'}\to_{h^{-1}} H^k \, \, \forall k\in \N.\]
such that 
\[(\t E_1 \t G_+^I) -(\t E_1 \t G_+^I)^c :L^r\to_{h^{0}} L^r,\ \ 0<r<\infty.\]
\end{lemma}

\begin{proof}First write
\begin{eqnarray}
\label{split E_1}
\tilde E_1 \t G_+^I = \Op (\t a_+^{-1}m_0 \t a_- \t a_+) \t G_+^I = \Op(\t a_+^{-1} m_0) \Op(\t a_- \t a_+) \t G_+^I + h\t e'_1(x',hD') \t G_+^I
\end{eqnarray}
for some $\t e_1' \in S^{0}_1(M_0)$. Note that
 \[\Op(\t a_-\t a_+) \t G^I_+ = \Op (\t a_-\t a_+) \gamma_*(\t \chi) \gamma_* G_+^M \gamma_*\chi =  \gamma_*(\t \chi) \Op (\t a_-\t a_+) \gamma_* G_+^M \gamma_*\chi.\]
We were able to commute multiplication by $\gamma_*\t \chi$ and  $\Op (\t a_-\t a_+)$ thanks to the fact that in \eqref{green on M} we have chosen $\t\chi$ so that $\gamma_*\t\chi$ is a function of $x_1$ only. Expanding $\t a_-\t a_+$ we see using \eqref{quantize the laplacian} that 
{\Small\begin{eqnarray}
\label{a_-a_+}\nonumber
\gamma_*(\t \chi) \Op (\t a_-\t a_+) \gamma_* G_+^M \gamma_*\chi&=& \gamma_*(\t \chi)\Big(({1+|df|^2}) - \gamma_*(h^2\partial_{y_1}^2 G_+^M) -2\Op\Big( \frac{ i-g_0(\xi',df)+ hi F}{1 + |df|^2}\Big) \gamma_*(h\partial_{y_1}G^M_+)\\ &&+  \t\rho_0(x',hD') \Op(S_1^1(M_0)) G_+^M+ h^2\Op(S^0_1)G_+^M\Big)\gamma_*\chi
\end{eqnarray}}
We first show that the first term of \eqref{split E_1} is a sum of an operator in $\Op(S^{-1}_1(M_0)) \Op(S^{0}_1(M))$ with an operator mapping $L^2 \to_{h^0} H^k$ and $L^{p'} \to_{h^{-1}} H^k$. To this end it suffices to show that \eqref{a_-a_+} is the sum of an operator in $\Psi^0_1(M)$ and an operator mapping $L^2 \to_{h^0} H^k$ and $L^{p'} \to_{h^{-1}} H^k$.

Using Lemma \ref{cut up G} we see that the last term is of the form 
$$h^2\gamma_*(\t \chi)\Big( \Op(S^{-2}_1(M)) +  \Op(S^{-\infty}(M))G_+^M\Op(S^{-\infty}(M))\Big)\gamma_*\chi.$$ 
This is the sum of a $\Psi$DO and a term which takes $L^2 \to_{h^0} H^k$ and $L^{p'}\to_{h^{-1}} H^k$ due to Proposition \ref{DosKenSal estimate} and Sobolev embedding. For the second last term, since $\t\rho$ is microlocally supported away from the characteristic set, it is of the form $\Op(S^{-2}_1(M)) + h^\infty\Op(S^{-\infty}(M))$ by Lemma \ref{cut up G}.

We analyze the term involving $h\partial_{y_1}G^M_+$ in \eqref{a_-a_+}. Using Lemma \ref{cut up G} we can write
{\small \[h\partial_{y_1} G^M_+ = \Op(S^{-1}_1(M)) + h \Op(S^{-\infty}(M)) G^M_+ \Op(S^{-\infty}(M)) + \Op(S^{-\infty}(M)) h\partial_{y_1}G_+^M \Op(S^{-\infty}(M)).\]}
The first term is a $\Psi$DO. The second term takes $L^2 \to_{h^0} H^k$ and $L^{p'}\to_{h^{-1}} H^k$ due to \eqref{DosKenSal estimate} and Sobolev embedding. The operator $h\partial_{y_1}G_+^M: L^2\to_{h^0} L^2$ since the Fourier multiplier $\frac{h\xi_1}{h^2\xi_1^2 - 2ih\xi_1 + h^2\lambda_l -1}$ is now uniformly bounded. Therefore, the term in \eqref{a_-a_+} involving $h\partial_{y_1} G^M_+$ can be written as an element of $\Psi^0_1(M)$ plus a term which takes $L^2 \to_{h^0} H^k$ and $L^{p'} \to_{h^{-1}} H^k$. Same argument shows that the term in \eqref{a_-a_+} involving $h^2\partial_{y_1}^2 G_+^M$ can be written as the linear combination of an operator in $\Psi^{0}_1(M)$ and a term mapping $L^2 \to_{h^0} H^k$ and $L^{p'} \to_{h^{-1}} H^k$.

We have thus shown that \eqref{a_-a_+} is the sum of an operator in $\Psi^0_1(M)$ with an operator mapping $L^2 \to_{h^0} H^k$ and $L^{p'} \to_{h^{-1}} H^k$. The second term of \eqref{split E_1} can be treated analogously to see that it is the sum of an operator in $\Op(S^0_1(M_0)) \Op(S^{-2}_1(M))$ with an operator mapping $L^2 \to_{h^0} H^k$ and $L^{p'} \to_{h^{-1}} H^k$. This completes the proof.
\end{proof}
\end{subsection}

\begin{subsection}{Approximate Semiclassical Inverse}\label{large freq}
Let $\tilde\Omega\subset M$ be a smooth bounded open subset contained in $M_+$ with a portion of the boundary intersecting $x_1 = 0$. Choose a bounded open interval $I \subset \R$ such that $\t \Omega \subset\subset \gamma(I\times M_0)$. Let $\tilde G_+^I$ be the Green's function defined by \eqref{green on M}, and $J^+ := J^{-1} {\bf 1}_{M_+}$. We first show that the operator 
\[
E_\ell :=(1- \tilde\rho (x',hD'))J^+ J \tilde G_+^I
\]
is a parametrix for $h^2\tilde\Delta_+$ in $\tilde \Omega$ for $\xi'$ large. We see first using \eqref{compact green estimate} and Proposition \ref{J inverse} that 

\begin{eqnarray}
\label{estimates for P_l}
E_\ell : L^2 \to_{h^{-1}} H^1,\ \ E_\ell : L^{p'}\to_{h^{-2}} H^1, \ \ E_\ell : L^{p'}\to_{h^{-2}} L^p.
\end{eqnarray}

We now state the parametrix property for $E_\ell$. If ${\bf 1}_{\t \Omega}$ is the indicator function of $\t\Omega$, then for all $v\in L^r(\t\Omega)$ we use ${\bf 1}_{\t\Omega} v$ to denote its trivial extension to a function in $L^r(M)$.
\begin{prop} \label{P_l remainders}
The operator $E_\ell$ is a Dirichlet parametrix. This means for $v\in L^{p'}$,
{\small \[
h^2{\bf 1}_{\t\Omega}\t \Delta_\phi E_\ell {\bf 1}_{\t\Omega} v= ((1+ |df|^2)(1 - \t \rho(hD'))(1+|df|^2)^{-1} +R_l + R_l')v,\ \ 
E_\ell v\mid_{M_-} = E_\ell v \mid_{x_1 = 0} = 0
\]}
as distribution on $\t \Omega$ with
\begin{eqnarray}
\label{large freq remainders}
R_l : L^2 \to _{h} L^2,\ \ R_l : L^{p'} \to_{h^{0}} L^2, \ \ R_l' : L^r\to_{h^0} L^r.
\end{eqnarray}
Furthermore, if ${\supp}(v) \subset \overline M_+$ then ${\supp}(R_l v) \subset \overline M_+$.
\end{prop}
\begin{proof} 
Express $h^2\t \Delta_+$ using \eqref{operator factorization} we get

{\begin{eqnarray}
\label{Delta P_l}
\nonumber
h^2(1+|df|^2)^{-1} \tilde \Delta_+E_\ell&=&( I-\tilde\rho(x',hD'))(1+|df|^2)^{-1} h^2\tilde \Delta_+ J^+ J  \tilde G^I_+ + [h^2\tilde\Delta_+, \tilde\rho] J^+J \tilde G^I_+ \\
&=& (I- \t\rho (x',hD')) (1+|df|^2)^{-1}+  [(1+|df|^2)^{-1}h^2\tilde\Delta_\phi, \tilde\rho] J^+J  \tilde G^I_+\\\nonumber  &&+ h\t E_1(I-J^+J) \t G_+^I + R
\end{eqnarray}}
where
{\small$$R = (I-\t\rho(x',hD'))(1+|df|^2)\Big(\t A_0(I - J^+J) -h^2 \t E_0(I - J^+J) + h^2 \Psi_1^0(M_0) + h^2 \Psi_1^0(M_0) J^+J\Big)\t G_+^I.$$}
Using \eqref{compact green estimate}, Sobolev embedding, and the fact that $I-\tilde\rho(x',hD')$ is microlocally disjoint from $\t A_0$ by the choice of $\t\rho$ in \eqref{the a's}, we see that every term in $R$ takes $L^2 \to _{h} L^2$, $L^{p'} \to_{h^{0}} L^2$. 

Directly by using Lemma \ref{composition estimate} the term $h\t E_1 \t G_+^I$ can be written as $R_l + R_l'$ where $R_l$ and $R_l'$ satisfies the estimates of \eqref{large freq remainders}. Writing explicitly the term 
$$ h\t E_1 \t J^+J\t G_+^I = h \t E_1 J^{-1} {\bf1}_{M_+} J \t G_+^I$$
we can commute $\t E_1$ with all the pseudodifferential operators by using standard calculus. Estimate the terms involving commutators using \eqref{compact green estimate} to see that they are of the form \eqref{large freq remainders}. Commuting $\t E_1$ with ${\bf1}_{M_+}$ yields nothing since $\t E_1 \in \Psi^{1}_1(M_0)$ and ${\bf 1}_{M_+}$ is constant along each fiber of the foliation of $M = \R\times M_0 $. Eventually $\t E_1$ will appear next to $\t G^I_+$ and we can use Lemma \ref{composition estimate} again to show that it is of the form \eqref{large freq remainders}.

The only remaining term to treat in \eqref{Delta P_l} is the $[(1+|df|^2)^{-1}h^2\tilde\Delta_+, \tilde\rho] J^+J  \tilde G^I_+$ term. This is done in
\begin{lemma}
\label{commuting the laplacian} The commutator term 
\[ [(1+|df|^2)^{-1}h^2\t \Delta_+, \t \rho(x',hD')] J^+ J \t G_+^I\]
maps $L^2\to_h L^2$ and $ L^{p'} \to_{h^{0}} L^2$.
\end{lemma}
and the proof is complete.\end{proof}

\begin{proof}[Proof of Lemma \ref{commuting the laplacian}]
Since $[hD_1, \t\rho(x',hD')] = 0$ we have, using the expression \eqref{quantize the laplacian}
{\small\begin{eqnarray*}
[(1+|df|^2)^{-1}h^2\t \Delta_\phi, \t \rho] J^+ J \t G_+^I &\equiv&  \Big [ \Op\Big( - 2\frac{ i-g_0(\xi',df)+ hi F}{1 + |df|^2} \Big), \t \rho(x',hD')\Big] hD_n J^+J\t G_+^I \\&&+ \Big[\Op\Big( \frac{1- |\xi'|^2+hi \xi'(K)}{1+ |df|^2}\Big), \t\rho(x',hD')\Big] J^+J\t G_+^I.\\
&\equiv&  \Big [ \Op\Big( - 2\frac{ i-g_0(\xi',df)+ hi F}{1 + |df|^2} \Big), \t \rho(x',hD')\Big] (I + \t A_+ J^{-1}){\bf 1}_{M_+}J\t G_+^I \\&&+ \Big[\Op\Big( \frac{1- |\xi'|^2+hi \xi'(K)}{1+ |df|^2}\Big), \t\rho(x',hD')\Big] J^+J\t G_+^I\\
\end{eqnarray*}}
Here "$\equiv$" denotes equivalence modulo a map taking $L^2\to_{h} L^2$ and $L^{p'} \to_{h^{0}} L^2$. Observe that since both $ \Big [ \Op\Big( - 2\frac{ i-g_0(\xi',df)+ hi F}{1 + |df|^2} \Big), \t \rho\Big] $ and $ \Big[\Op\Big( \frac{1- |\xi'|^2+hi \xi'(K)}{1+ |df|^2}\Big), \t\rho(x',hD')\Big]$ are $\Psi$DO on $M_0$ they commute with the indicator function ${\bf 1}_{M_+}$. Using this and the fact that $\t\rho(x',hD')(x',\xi')$ is supported away from the characteristic set of $\t \Delta_+$ we have 
\[[(1+|df|^2)^{-1}h^2\t \Delta_\phi, \t \rho] J^+ J \t G_+^I \equiv h(I + \t A_+J^{-1}) {\bf 1}_{M_+} a(x,hD) \t G^I_+ + h J^+ b(x,hD) \t G^I_+\]
for some $a,b \in S^1_0(M)$ supported away from the characteristic set of $\t \Delta_+$. We now apply Lemma \ref{microlocal wf} to see that this operator takes $L^2\to_{h} L^2$ and $L^{p'} \to_{h^{0}} L^2$.
\end{proof}
The following Lemma says that $E_\ell$ is almost like $\t G_+^I$ on compact subsets of the open set $M_+$:
\begin{lemma}
\label{P_l disjoint support}
Let $a(x,hD)$ be a first order differential operator with coefficients compactly supported in the region $\{x_1 \geq \epsilon \}$ for some $\epsilon >0$. Then
\[{\bf 1}_{\t \Omega} ha(x,hD)(\t G_+^I - E_\ell){\bf 1}_{\t \Omega}  : L^2(\t\Omega)\to_h L^2(\t\Omega),\]
\[{\bf 1}_{\t \Omega} ha(x,hD)(\t G_+^I - E_\ell){\bf 1}_{\t \Omega}  : L^{p'}(\t\Omega)\to_{h^0} L^2(\t\Omega).\]
\end{lemma}
\begin{proof}
We have by definition 
$${\bf 1}_{\t \Omega} ha(x,hD) (\t G_+^I - E_\ell){\bf 1}_{\t \Omega} = {\bf 1}_{\t \Omega} ha(x,hD) (\t G_+^I - (1-\t\rho(x',hD'))J^+J \t G_+^I){\bf 1}_{\t \Omega}.$$
For the $ {\bf 1}_{\t \Omega} ha(x,hD) (\t G_+^I - J^+J \t G_+^I){\bf 1}_{\t \Omega}$ portion we have

\[{\bf 1}_{\t \Omega} ha(x,hD) (I - J^+J)\t G_+^I {\bf 1}_{\t \Omega}  ={\bf 1}_{\t \Omega}ha(x,hD) J^{-1}{\bf1}_{ M_-} J\t G_+^I {\bf 1}_{\t \Omega} \]
By assumption $a(x,hD)$ is a first order differential operator with coefficients supported in $\{x_1 \geq \epsilon >0\}$. The proof is complete by evoking Lemma \ref{disjoint support for J inverse} and the mapping properties of \eqref{compact green estimate}.

Moving on to the $ {\bf 1}_{\t \Omega} ha(x,hD) (\t \rho(x',hD') J^{-1}{\bf 1}_{M_+}J \t G_+^I){\bf 1}_{\t \Omega}$ portion we have
$${\bf 1}_{\t \Omega} ha(x,hD) (\t \rho(x',hD') J^{-1}{\bf 1}_{M_+}J \t G_+^I){\bf 1}_{\t \Omega} \equiv{\bf 1}_{\t \Omega} ha(x,hD) (J^{-1}{\bf 1}_{M_+}J \t \rho(x',hD')  \t G_+^I){\bf 1}_{\t \Omega}$$
where $\equiv$ denotes equality up to a map taking $L^2(\t\Omega)\to_h L^2(\t\Omega)$ and $L^{p'}(\t\Omega)\to_{h^0} L^2(\t\Omega)$. Note that we were able to commute $\t\rho(x',hD')$ with ${\bf 1}_{M_+}$ because the indicator function is constant along each fiber $\{x_1 = const\}$ and $\t\rho(x',hD')$ acts in the $x'$ direction only. The proof is complete by observing that \eqref{supp of rho tilde} says that $\t\rho(x',\xi')$ is supported away from the characteristic set of $\t\Delta_+$ and apply Lemma \ref{microlocal wf}.
\end{proof}

At small $\xi'$ on the support of $\t\rho(x',\xi')$ the square root defined in \eqref{r_0} is discontinuous so we cannot factor $\t\Delta_+$ as in \eqref{decomposition}. Here we are saved by the fact that $\t\Delta_+$ is actually elliptic thanks to \eqref{supp of rho tilde}. The parametrix in this region can then be constructed via is straightforward elliptic calculus. To this end define
\[
{\mathcal P}(x',\xi) := \xi_1^2 -2\xi_1\frac{(i-g_0(df,\xi')-i hF)}{ 1+ |df|^2} -\frac{(1- |\xi'|^2 -ih\xi'(K))}{1+|df|^2}
\]
and 
\[
E_s := \frac{\t \rho}{{\mathcal P}}(x',hD) \circ (1+|df|^2)^{-1}.
\]
The parametrix $E_s$ inverts $h^2\tilde{\Delta}_{+}$ at small $\xi'$ modulo $O(h)$:

\begin{prop}
\label{small freq parametrix} 
We have that for $0<r<\infty$, $E_s: L^r \to W^{2,r}$. Furthermore, at small frequencies $E_s$ inverts $h^2\t \Delta_+$ in the sense that
\[
h^2 \t \Delta_+ E_s =  (1+|df|^2) \t\rho(x', hD') (1+|df|^2)^{-1}+R_s
\]
for some $R_s : L^r\to_h L^r$, for all $0<r<\infty$.
\end{prop}

\begin{proof}
Standard symbol calculus defined in Section 2 does not apply as $1/{\mathcal P}(x',\xi)$ is not a proper symbol, due to the zeros of ${\mathcal P}(x',\xi)$. 

We therefore write $\t \rho(\xi')/{\mathcal P} (x',\xi) $ as
\[
 (1 - \chi_{3}(\xi))\t \rho(\xi')/{\mathcal P}(x',\xi) + \chi_{3}(\xi)\t \rho(\xi')/{\mathcal P}(x',\xi)
\]
where $\chi_{3}(\xi) \in S^{-\infty}(M)$ is a smooth cutoff which vanishes in $\{|\xi| \geq 3\}$, and $\chi_3(\xi) = 1$ in $\{|\xi| \leq 2\}$.  

Thanks to \eqref{supp of rho tilde}, ${\mathcal P}(x',\xi)$ does not vanish on $\supp(\t \rho(\xi'))$, and we can deduce that $\chi_{3}(\xi) \t \rho(\xi')/{\mathcal P} (x',\xi) \in S^{-\infty}(M)$. Moreover, since ${\mathcal P}(x',\xi)$ fails to be elliptic only inside the set where $\chi_{3} \equiv 1$, we have that $(1 - \chi_{3}(\xi))/{\mathcal P}(x',\xi) \in S^{-2}_1(M)$.  

Therefore $\frac{\t \rho}{{\mathcal P}} (x',hD)$ is understood to be the linear combination of an operator associated to a symbol in $S^{-\infty}(M)$ and an operator with symbol in $S^{-\infty}S^{-2}_1(M)$. Proposition \ref{sobolev mapping} can now be evoked to conclude that $E_s : L^r \to W^{2,r}$ and Proposition \ref{sobolev composition} asserts that  
\[
h^2(1+|df|^2)^{-1}\tilde\Delta_+\mathrm{Op}\left( \frac{\t \rho}{{\mathcal P}} \right) = \mathrm{Op}((1 - \chi_{3})\t \rho) + \mathrm{Op}(\chi_{3}\t \rho) + hR_{-1} = \mathrm{Op} (\t \rho) + R_s.
\]
\end{proof}

It turns out that $E_s$ preserves support in $M_+$.

\begin{prop}\label{SmallFrequencySupport}
If $v \in L^r(M)$ with $r\in (1,\infty)$ has support contained in the closure of $M_+$ then both $E_sv$ and $R_s v$ are supported in $\overline M_+$, where $R_s$ is as in Proposition \ref{small freq parametrix}. In particular the trace of $E_sv$ on $\{x_0 = 0\}$ vanishes.
\end{prop}

\begin{proof}
Let $U\subset M_0$ be a coordinate patch. It suffices to prove this statement for compactly supported smooth functions $v$ in the (infinite) strip $\R\times U = \{(x_1, x') \mid x'\in U\}$. Let $v(x_1,x')$ also denote the pullback function by the coordinate map then $\mathrm{Op}\left( \frac{\t \rho}{{\mathcal P}} \right)v(x_1,x')$ is
\begin{equation}\label{Oprhopph}
h^{-n}\int_{ \R^{n-1}}\int_{\R^{n-1}}e^{i\xi' \cdot (x'-y')/h}\int_{-\infty}^\infty \int_{-\infty}^{\infty}\frac{\t \rho(\frac{x'+y'}{2},\xi')}{{\mathcal P}(\frac{x'+y'}{2},\xi)}{v}(y_1,y')e^{i\xi_1 (x_1-y_1)/h} \, d\xi_1 dy_1 \, d\xi'dy'.
\end{equation}
We want to evaluate the inner most integral in $\xi_1$ using contour integral in $\C$. 
Since $e^{i\xi_1 (x_1-y_1)/h}$ is holomorphic, we need to find the poles of $\frac{1}{{\mathcal P}(\frac{x'+y'}{2},\xi)}$ as a polynomial in $\xi_1$. Factoring and suppressing the dependence on the spacial variable, we have
\[
{\mathcal P}(\frac{x'+y'}{2}, \xi) = -(\xi_1 - a_+)(\xi_1 - a_-)
\]
where $a_{\pm} = i\frac{(1+ig_0(df,\xi') + hF )\pm r_0}{1+|df|^2}$ with $r_0$ the square root give by \eqref{r_0}.

Therefore ${\mathcal P}(x',\xi)$, viewed as a polynomial in $\xi_1$, has two roots: $a_+$ and $a_-$. The symbol $a_+$ has positive imaginary part because $r_0$ is defined using the branch of the $\sqrt{\cdot}$ with cut along the negative real axis.  

We want to ensure that on the support of $\t\rho(\frac{x'+ y'}{2}, \xi')$ the imaginary part of $a_{-}(\frac{x'+y'}{2}, \xi')$ is strictly positive for all $h >0$ small, $x'$, $y'$, and $\xi'$. First, by \eqref{supp of rho tilde} the polynomial ${\mathcal P}(\frac{x'+y'}{2}, \xi) $ never vanishes on the support of $\t\rho(\frac{x'+ y'}{2}, \xi')$ so the imaginary part of  $a_-$ is bounded away from zero on the support of $\t\rho(\frac{x'+ y'}{2}, \xi')$ (otherwise $\xi_1$ can be chosen to make ${\mathcal P}(\frac{x'+y'}{2}, \xi)$ close to zero, contravening \eqref{supp of rho tilde}). This means that on the support of $\t\rho(\frac{x'+ y'}{2}, \xi')$, the real-valued function $1 +hF - {\rm Re}(r_0)$ stays uniformly away from zero. Note that the standard branch of $\sqrt{\cdot}$ defined on $\C \backslash \{z\in \C\mid {\rm Re}(z) \leq 0\}$ has a continuous extension onto the closed blownup manifold $[\C ; \{z\in \C\mid {\rm Re}(z) \leq 0\}] \to \C$. This means that for each fixed $h>0$ small, $x'\in M_0$, and $y'\in M_0$, the function $1 +hF - {\rm Re}(r_0)$ is either uniformly positive or uniformly negative for all $\xi'$. Choosing $\xi' = 0$, we see that $1 +hF - {\rm Re}(r_0)$ is uniformly positive. Therefore we have that ${\rm Im}(a_-) >0$ on the support of $\t\rho(\frac{x'+ y'}{2}, \xi')$ as well.

Therefore to evaluate the $\xi_1$ integral of \eqref{Oprhopph} we must use the contour integral in the upper-half of $\C$. Doing so we get  
\[
2\pi i \t \rho(\frac{x'+y'}{2},\xi')\int_{-\infty}^{x_1}\frac{v(y_1,y')(e^{ia_- (x_1-y_1)/h}-e^{ia_+ (x_1-y_1)/h})}{(a_+ - a_-)} \, dy_1
\]
for the case when $a_+ \neq a_-$. So \eqref{Oprhopph} can be written as
{\Small\begin{equation}\label{SupportForm}
2\pi i h^{-n}\int_{\R^{n-1}}\int_{\R^{n-1}}e^{i\xi' \cdot (x'-y')/h} \int_{-\infty}^{x_1}\frac{\t \rho(\xi')v(y_1,y')(e^{ia_- (x_1-y_1)/h}-e^{ia_+ (x_1-y_1)/h})}{(a_+ - a_-)} \, ds \, d\xi'dy'.
\end{equation}}
We now treat the case when $a_+$ is close to $a_-$. In the case when $a_+ = a_-$ the integral vanishes by residue calculus. In a small neighbourhood of this set we have
\[
\lim_{a_+ - a_- \rightarrow 0}\frac{e^{ia_- (x_1-y_1)/h}-e^{ia_+ (x_1-y_1)/h}}{(a_+ - a_-)} = \frac{i(x_1-y_1)}{h}e^{ia_- (x_1-y_1)/h}.
\]
Thus \eqref{SupportForm} is finite, and so if $v \in C^{\infty}_c(\R \times U)$ is supported only in $\overline M_+$, it is clear that 
\begin{equation}\label{SupportProp}
 \mathrm{Op}(\t \rho/{\mathcal P}) v(x_1,x') = 0 \mbox{ for } x_1 \leq 0.
\end{equation}
Using Proposition \ref{small freq parametrix} and boundedness of the trace operator we see that if $v\in L^r(M)$ is supported in $\overline M_+$ then $ \mathrm{Op}(\t \rho/{\mathcal P}) v(x_1,\cdot) = 0$ for $x_1 \leq 0$.

One can see that $R_s$ is supported in $\overline M_+$ by writing 
\[
(1+|df|^2)^{-1}h^2 \t \Delta_+ E_s - \t \rho(h D')(1+|df|^2)^{-1} =  hR_s
\]
and observe that the left side is supported in $\overline M_+$.

\end{proof}

\end{subsection}

\begin{subsection}{Proof of Proposition \ref{green's function}}\label{graph case}

We now turn the semiclassical parametrix constructed in Subsection \ref{large freq} into a proper inverse for $h^2\t\Delta_+$. 
By Propositions \ref{P_l remainders} and \ref{small freq parametrix} ${\bf 1}_{\t \Omega} (E_s + E_\ell) {\bf 1}_{\t \Omega}$ is a parametrix for $h^2\t\Delta_+$ in $\t \Omega$. In the semiclassical limit the remainder terms of the parametrix is sufficiently small so that one can modify it to become a resolvent of $h^2\Delta_+$.

We begin with the case where $\Gamma$ can be flattened by using coordinates given by the graph of a smooth function. Let $\Omega$ be a bounded domain in $ M$ contained in the epigraph $\{y_1 > f(y')\}$ of a smooth function $f \in C^\infty(M_0)$ and a portion of the boundary $\Gamma \subset \partial \Omega$ is contained in the graph. Use the change of variable $\gamma : (y_1,y') \mapsto (x_1 = y_1 - f(y'), x' = y')$ and define $\t \Omega := \gamma(\Omega)$. In these new coordinates $\t\Omega \subset M_+$ and $\gamma(\Gamma) \subset \{x_1 = 0\}$.

\begin{prop}
\label{graph green}
There exists a Green's function
\[G_\Gamma: L^2(\Omega)\to_{h^{-1}} L^2(\Omega),\ G_\Gamma : L^{p'}(\Omega)\to_{h^{-2}} L^p(\Omega).\] 
which satisfies the relation $h^2\Delta_+ G_\Gamma = Id$ as distributions on $\Omega$. It has the explicit representation
\[G_\Gamma = \gamma^*\big({\bf 1}_{\t \Omega}(E_s + E_\ell) {\bf 1}_{\t \Omega}  (I + R) \big)\] with the remainder $R$ having the asymptotic as $h\to 0$ given by
\[ R : L^{p'} (\t\Omega) \to_{h^0} L^2(\t \Omega),\ \ \ R : L^{2} (\t\Omega) \to_{h} L^2(\t \Omega).\]
Furthermore, $G_\Gamma v\in H^1(\Omega)$ for all $v \in L^{p'}$ with vanishing trace on $\Gamma$.
\end{prop}
\begin{proof}
Change coordinates $(y_1,y') \mapsto (x_1,x')$ so that $\t \Gamma := \gamma(\Gamma)\subset \{x_1 = 0\}$ and let $\tilde \Delta_+$ be the pulled-back conjugated Laplacian. By Proposition \ref{P_l remainders} and Proposition \ref{small freq parametrix}, 
$$h^2\tilde\Delta_+  {\bf 1}_{\t \Omega} (E_s + E_\ell) {\bf 1}_{\t \Omega} = I + R_s + R_l + R_l'$$
with $R_s + R_l'$ mapping $L^r(\t\Omega) \to_{h} L^r(\t \Omega)$. Let $S :L^r(\t\Omega)\to L^r(\t\Omega)$ denote the inverse of $(1 +  R_s + R_l')$ by Neumann sum. This yields in $\t \Omega$ 
\[h^2\tilde\Delta_+ {\bf 1}_{\t \Omega}(E_s + E_\ell){\bf 1}_{\t \Omega}S = I + R_l S.\]
By Proposition \ref{P_l remainders}
we have $R_lS : L^2(\t\Omega)\to_{h} L^2(\t\Omega)$ while $R_lS : L^{p'}(\t\Omega) \to_{h^0} L^2(\t\Omega)$. Therefore we can invert by Neumann series again to obtain a right inverse for $h^2\t\Delta^+$ of the form ${\bf 1}_{\t \Omega}(E_s + E_\ell){\bf 1}_{\t \Omega}S(I+ R_lS)^{-1}$ where 
$$(I + R_l S)^{-1} : L^2\to_{h^0} L^2,\ \ (I + R_l S)^{-1} : L^{p'} \to_{h^0} L^2 + L^{p'}.$$
Changing variables we see that $h^2\Delta_+ G_\Gamma = Id$ by setting 
$$G_\Gamma := \gamma^*\big({\bf 1}_{\t \Omega}(E_s + E_\ell) {\bf 1}_{\t \Omega} S(1 + R_l S)^{-1}\big).$$
The mapping properties and Dirichlet boundary condition follows then from the analogous properties for $E_\ell$ and $E_s$ outlined in Propositions \ref{small freq parametrix}, \ref{SmallFrequencySupport}, \ref{P_l remainders}, and \eqref{estimates for P_l}.
\end{proof}
To prove Proposition \ref{green's function} in the general case, we patch together Green's functions as \cite{chung-tzou}. Let $\Gamma$ be a closed and connected component of $\partial \Omega$ contained in an open set $\Omega_\Gamma \subset M$ such that there exists $f\in C^\infty(M_0)$ for which $\Omega_\Gamma \cap \partial M \subset \{y_1 \geq f(y')\}$ and $\partial M \cap \Omega_\Gamma \cap \{ y_1 = f(y')\} = \Gamma$. We may choose $\Omega_\Gamma$ small enough such that $\Omega_\Gamma\cap \Omega$ is contained in the epigraph of $f$. If a compact connected component of $\partial \Omega$ satisfies this condition, we say that {\em $\Gamma$ is compatible with a smooth function $f$}.

Choose $\chi\in C^\infty_c(\Omega_\Gamma)$ such that $\chi =1$ on $\Gamma$ and define ${\mathcal O} := \Omega_\Gamma \cap \{y_1> f(y')\}$. By the fact that $(\Omega_\Gamma \cap \partial M)\backslash \Gamma$ lies strictly above the graph $y_1 = f(y')$, we can arrange $\chi$ so that
\begin{eqnarray}
\label{support of Dchi}
\exists \epsilon > 0 \mid {\supp}({\bf 1}_{\Omega} D\chi) \subset\{ (y_1,y')\mid y_1 \geq f(y') + \epsilon\}.
\end{eqnarray}
Choose $I\subset \R$ such that $\Omega$ is contained in $I\times M_0$ and let $G_+^I$ be the Green's function defined via \eqref{green on M}. Let $G_\Gamma$ now the Green's function constructed in Proposition \ref{graph green} for the domain $\mathcal O$ with vanishing Dirichlet condition along the portion $\{y_1 = f(y')\}$.

Now let $\Pi_\Gamma := \chi {\bf 1}_\Omega (G_+^{I} - G_\Gamma){\bf 1}_{\mathcal O}$ and see that it satisfies the estimates
\begin{eqnarray}\label{Pi estimates}
\Pi_\Gamma : L^{p'}(\Omega) \to_{h^{-2}} L^p(\Omega),\ \ \Pi_\Gamma : L^{2}(\Omega) \to_{h^{-1}} H^1(\Omega).\end{eqnarray}

Proposition \ref{graph green} yields that $\Pi_\Gamma v \in H^1(\Omega)$ and 
\begin{eqnarray}
\label{same trace} ((\Pi_\Gamma v)-  (G_+^I v))\mid_{\Gamma}= 0
\end{eqnarray}
for all $v\in L^{p'}(\Omega)$.
\begin{lemma}
\label{remainder of difference in green}
One has the estimates 
\[
h^2\Delta_+ {\bf 1}_\Omega\Pi_\Gamma : L^{p'}(\Omega)  \to_{h^0} L^2(\Omega),\ \ h^2\Delta_+ {\bf 1}_\Omega\Pi_\Gamma  : L^{2}(\Omega) \to_{h^1} L^2(\Omega).
\]
\end{lemma}
Assuming this Lemma, let $\Omega$ be a domain in $M$ with smooth boundary and $\Gamma\subset \partial\Omega$ be a compact set contained in $\{y\in \partial\Omega \mid g( \partial_{y_1} ,\nu(y))>0\}$. Since $\Omega \subset I\times M_0$ for some simple manifold $M_0$ we may write $\Gamma $ as the disjoint union $\bigcup\limits_j \Gamma_j$ of connected compact components $\Gamma_j$ each of which is compatible with a smooth function $f_j$.

For each $\Gamma_j$ let $\Pi_{\Gamma_j}$ and $\chi_j$ be 
as before. By \eqref{same trace} we have
\[\Big( G^I_+ v- \sum_{j = 1}^k \Pi_{\Gamma_j} v\Big)\mid_\Gamma = 0\]
for every $v\in L^{p'}(\Omega)$. Lemma \ref{remainder of difference in green} yields that  $h^2\Delta_+ {\bf 1}_\Omega\big(G^I_+ - \sum_{j = 1}^k \Pi_{\Gamma_j} \big){\bf 1}_{\Omega}$ is identity plus an operator $R'$ which takes 
$ L^2(\Omega)\to_{h} L^2(\Omega)$ and $L^{p'}(\Omega)\to_{h^0} L^{2}(\Omega)$.
Observe that we can as before find an inverse $(1+ R')^{-1} : L^2 \cup L^{p'} \to_{h^0} L^2$. Proposition \ref{green's function} is now complete by the estimates of \eqref{Pi estimates} and \eqref{DosKenSal estimate}. 
All that remains is:

\begin{proof}[Proof of Lemma \ref{remainder of difference in green}]
By Proposition \ref{graph green} we have $G_{\Gamma}$ is a right inverse for $h^2\Delta_+$ in $\mathcal O$, and $\supp(\chi {\bf 1}_\Omega) \subset \mathcal O$, so $\chi h^2 \Delta_+ {\bf 1}_{\Omega} G_{\Gamma}v(y) = \chi v(y)$ in $\Omega$. By construction $h^2 \Delta_+G_+^I = Id$ in $ \Omega$, so $h^2\Delta_+ {\bf 1}_{\Omega} G^I_+v = v$. Therefore as a distribution acting on $C^\infty_c(\Omega)$, the only term in $h^2\Delta_+\Pi_\Gamma v(y)$ is $[h^2\Delta_+, \chi_j(y)] {\bf 1}_\Omega(G_+^I - G_{\Gamma_j}){\bf 1}_{\mathcal O}v(y)$. The operator $\Pi_\Gamma$ is defined using finitely many partitions so we may assume that the sum contains only one term given by $\chi_j = \chi$.

To treat this we make a change of variable by $(y_1,y')\mapsto (x_1 = y_1 - f(y'),x' = y' )$ and denote with a tilde the quantities obtained by pushing forward with this coordinate change. Using Proposition \ref{graph green} we see that this term is of the form
\[
[h^2\t\Delta_+, \t\chi(x)]{\bf 1}_{\t \Omega} (\t G_+^I - (E_s + E_\ell) {\bf 1}_{\t \Omega} (I + R)){\bf 1}_{\t{\mathcal O} }
\]
where $R$ takes $ L^{p'}(\t\Omega) \to_{h^0} L^2(\t\Omega$ and $L^2(\t\Omega)\to_{h} L^2(\t\Omega)$.
Compute $[h^2\t\Delta_+, \t\chi]$ explicitly and use the estimates in Proposition \ref{small freq parametrix} and \eqref{estimates for P_l} we see that 
{\Small\begin{equation}\label{LastMultiGraphTerm}
[h^2\t\Delta_+, \t\chi]{\bf 1}_{\t \Omega} (\t G_+^I - (E_s + E_\ell) {\bf 1}_{\t \Omega} S(1 + R_l S)^{-1} ) {\bf 1}_{\t{\mathcal O} } = [h^2\t\Delta_+, \t\chi]{\bf 1}_{\t \Omega} (\t G_+^I - (E_s + E_\ell)) {\bf 1}_{\t \Omega}S {\bf 1}_{\mathcal O} + E 
\end{equation}}
where $E$ is an error with estimates $E: L^{p'}(\t \Omega) \to_{h^0} L^2(\t \Omega)$ and $E: L^{2}(\t \Omega) \to_{h^1} L^2(\t \Omega)$.
This error has the desirable estimates so it remains only to analyze the first term of \eqref{LastMultiGraphTerm} given by $[h^2\t\Delta_+, \t\chi]{\bf 1}_{\t \Omega} (\t G^I_+ - (E_s + E_\ell)) {\bf 1}_{\t \Omega}S  {\bf 1}_{{\mathcal O}}$.

As distributions acting on $C^\infty_c(\t\Omega)$, the first order differential operator  $[h^2\t\Delta_+ \t\chi]$ commutes with the indicator function ${\bf 1}_{\t \Omega}$ which gives us
\[[h^2\t\Delta_+, \t\chi]{\bf 1}_{\t \Omega} (\t G_+^I- (E_s + E_\ell)) {\bf 1}_{\t \Omega}=
{\bf 1}_{\t \Omega} [h^2\t\Delta_+, \t\chi](\t G_+^I - E_\ell){\bf 1}_{\t \Omega}. -  {\bf 1}_{\t \Omega} [h^2\t\Delta_+, \t\chi]E_s{\bf 1}_{\t \Omega}.
\]
Now $E_s$ maps $L^2\to_{h^0}H^2$ and $L^{p'}\to W^{2,p'} \hookrightarrow_{h^{-1}} H^1$. This is composed with the commutator $[h^2\t\Delta_+, \t\chi]$ which maps $H^1$ to $L^2$ with the gain of $h$. As such the term involving $E_s$ has the desirable asymptotic as $h\to 0$.  The remaining term to treat is ${\bf 1}_{\t \Omega} [h^2\t\Delta_+, \t\chi](\t G_+^I - E_\ell) {\bf 1}_{\t \Omega} $. By \eqref{support of Dchi} the commutator ${\bf 1}_{\t \Omega}[h^2\t\Delta_+, \t\chi] $ is a first order differential operator whose coefficients vanishing in $\{x_1 \leq \epsilon\}$. Lemma \ref{P_l disjoint support} can then be applied to give us the necessary estimates for this term.
\end{proof}
\end{subsection}
\end{section}

\begin{section}{Complex Geometrical Optics }
Let $M = \R \times M_0$ and $g = dy_1^2 + g_0$ be a metric on $M$. Consider the bounded domain $\Omega\subset  M$ and let $\Gamma \subset \partial \Omega$ be an open set compactly contained in $\{y\in \partial \Omega\mid g (\nu(x),\partial_{y_1})>0\}$ where $\nu_n $ the outward normal. By Proposition \ref{green's function} there exists a resolvent $G_\Gamma$ for $h^2\Delta_\phi$ whose trace along$\Gamma$ vanishes and \[G_\Gamma: L^2(\Omega) \to_{h^{-1}} L^2(\Omega),\ \ \ G_\Gamma : L^{p'}(\Omega)\to_{h^{-2}} L^{p}(\Omega).\]

\begin{subsection}{Application of Green's Function to Solvability}

In the geometric setting described above, we can use the same argument as in \cite{DosKenSal} to prove the following \begin{prop}
\label{solve for rhs}
Let $L\in L^2$ satisfy $\|L\|_{L^2} \leq Ch^2$, and let $q\in L^{n/2}(\Omega)$. If $a = a_h\in L^\infty$ is a uniformly bounded family of functions in $h$, one can find a solution to
\begin{eqnarray}
e^{-y_1/h} h^2(\Delta_g + q) e^{y_1/h}r = h^2 q a + L\ \ \ r\mid_{\Gamma} = 0.\end{eqnarray}
The solution $r$ satisfies the asymptotic $\|r\|_{L^2} \leq o(1)$ and $\|r\|_{L^p} \leq O(1)$ as $h\to0$.
\end{prop}
Observe that we can generalize this to metrics which are conformal to $dy_1^2\oplus g_0$ (i.e. CTA metrics). Indeed, if $c^{-1}g$ is a metric conformal to $g = dy_1^2 \oplus g_0$ then one can write as in \cite{dksu} the Schr\"odinger operator for $c^{-1}g$ as
\[c^{\frac{n+2}{4}} (\Delta_g + q) u =( \Delta_{c^{-1}g} + q_c )(c^{\frac{n-2}{4}}u ) \]
where $q_c := cq + c^{\frac{n+2}{4}}\Delta_g c^{-\frac{n-2}{4}} \in L^{n/2}(\Omega)$. Therefore Proposition \ref{solve for rhs} immediately generalizes to metrics which are conformal to $dy_1^2 \oplus g_0$. 
\begin{cor}
\label{solve for rhs cor}

Let $g$ be a CTA metric on $M = \R\times M_0$. Let $\Omega \subset M$ be a bounded open subset and $\Gamma\subset\subset \{y\in \partial M \mid g(\partial_{y_1}, \nu(y)) >0\}$. For all $L \in L^2(\Omega)$ with $\|L\|_{L^2} \leq Ch^2$, $q\in L^{n/2}(\Omega)$, and $a = a_h\in L^\infty$ with $\|a_h\|_{L^\infty}$ uniformly bounded as $h\to 0$, there exists a solution of 
\begin{eqnarray}
\label{solve}
e^{-y_1/h} h^2(\Delta_{g} + q) e^{y_1/h}r = h^2 q a + L\ \ \ r\mid_{\Gamma} = 0\end{eqnarray}
with asymptotic given by $\|r\|_{L^2} \leq o(1)$ and $\|r\|_{L^p} \leq O(1)$ as $h\to 0$.
\end{cor}

\end{subsection}
\begin{subsection}{CGO In Conformally Transversally Anisotropic Manifold}
We first construct the CGO ansatz following the method of \cite{ksu}. Assume that $\Omega\subset \R\times \Omega_0$ where $\Omega_0 \subset M_0$ is a simple manifold compactly contained in a slightly larger simple manifold $\hat \Omega_0$. Let $\omega \in \hat \Omega_0\backslash\Omega$ and set $(t,\theta)$ to be spherical coordinate around this point. The metric in these coordinates is $g =c( dy_1^2 \oplus dt^2 \oplus g_0')$ so $\phi + i\psi = y_1 + it$ solves the eikonal equation
\[g(d \phi + id\psi, d\phi + id\psi) = 0.\]
We can solve the transport equation $g(d \phi + i \psi,da ) + \Delta_g (\phi+i \psi) = 0$, by setting 
\begin{eqnarray}
\label{amplitude}
a = |g|^{-1/4}  e^{i(y_1 + it)\lambda} \beta (\theta)
\end{eqnarray}
where $\beta$ is any smooth function on $S^{n-2}$. With $\phi$, $\psi$, and $a$ chosen as such we have 
\[e^{-(\phi + i\psi)/h} \Delta_g e^{(\phi + i\psi)/h} a = O_{L^2}(h^2).\]
We now need to construct a reflection term $e^{\ell/h}b$ which kills $e^{(\phi + i\psi)/h} a$ on $\Gamma$. As before we will construct $\ell$ supported near $\Gamma$ solving the approximate equation
\begin{eqnarray}
\label{approx eikonal}
g(d \ell ,d\ell)\mid_{z} = dist(z, \Gamma)^\infty
\end{eqnarray}
with boundary condition $\ell\mid_{\Gamma} = (\phi + i \psi)\mid_{\Gamma}$ and $ \partial_\nu\ell\mid_{\Gamma} = -\partial_\nu(\phi + i \psi)\mid_{\Gamma}$. 
The construction will be localized so we may assume without loss of generality that $\Gamma$ is compactly contained in single connected component of $\{ y\in \partial \Omega \mid g(\nu(y), \partial_{y_1}) >0\}$. Using boundary normal coordinates $(s,z') \in \R_+ \times \partial M$ near $\Gamma$, we may express the metric as $ds^2 \oplus g'$ for some symmetric two tensor $g'$ which annihilates $\partial_s$ where $s$ is the distance away from the boundary. Note that in a small neighbourhood of $\Gamma$, $g(ds, d\phi) \geq \epsilon$ which allows us to solve for individual terms of the formal expansion 
\[\ell(s,z') = \sum\limits_{j = 0}^\infty s^j \ell_j(z'),\ \ \ell_0(z') = \phi(0,z') + i\psi(0,z'),\ \ \ell_1(z') = -\partial_\nu(\phi + i\psi)(0,z')\]
so that \eqref{approx eikonal} is satisfied. Borel Lemma can then be employed to construct $\ell(s,z')$ solving \eqref{approx eikonal}. Similarly, the approximate transport equation 
\[g(d\ell, db) = d(z,\partial M)^\infty \ \ \ b \mid_\Gamma= -a\mid_\Gamma\]
can also be solved iteratively using formal power series and the fact $|g(ds, d\ell) |\geq \epsilon$ near $\Gamma$. Since we are only interested in the behaviour of $b$ at and near $\Gamma$, we may construct it to be supported in a small neighbourhood of $\Gamma$.

The ansatz given by $e^{\ell/h} b$ satisfies $e^{-\phi/h} h^2\Delta_g (e^{\ell/h} b) = e^{(\ell - \phi)/h}(O(dist(z,\Gamma)^\infty + O_{L^\infty}(h^2))$. Note that by the boundary condition for $\partial \ell$ and the fact that $\Gamma\subset \subset \{y\in \partial \Omega \mid g(\nu(y), \partial_{y_1}\}$, we have the comparison $\frac{1}{C} d(z,\Gamma)\leq \phi(z) - \ell(z) \leq C d(z,\Gamma)$ for $z$ on the support of $b$ . So by analyzing $d(z,\Gamma) \leq \sqrt{h}$ and $dist(z,\Gamma) \geq \sqrt{h}$ we have that 
$$e^{-\phi/h} h^2\Delta_g (e^{\ell/h} b) = e^{(\ell - \phi)/h}O_{L^\infty}(h^2),\ \ e^{\ell/h}b \mid_{\Gamma} = -e^{(\phi+i\psi)/h}a\mid_\Gamma.$$
We have constructed an ansatz of the form 
\begin{eqnarray}
\label{form of anstaz}
u_{\rm{ans}} = e^{(\phi+i\psi)/h} a + e^{\ell/h} b = e^{(\phi + i\psi)/h}(a + a_h)
\end{eqnarray}
where $a$ is of the form \eqref{amplitude} and $a_h$ satisfies $\|a_h\|_{L^\infty} \leq C$ and $a_h \to 0$ pointwise such that 
\begin{eqnarray}
\label{ansatz}
 h^2(\Delta_g +q) u_{\rm{ans}} = h^2  e^{(\phi+i\psi)/h}(q(a+a_h) + L),\ \ u_{\rm{ans}}\mid_{\Gamma} = 0
\end{eqnarray}
with $\|L\|_{L^2} \leq Ch^2$. 

Following precisely the argument of Prop 3.4 in \cite{DosKenSal}, the ansatz $u_{\rm{ans}}$ combined with Corollary \ref{solve for rhs} allows us to construct CGO which are the key ingredients to solving inverse problems. The only difference is that thanks to the fact that $G_\Gamma$ satisfies the Dirichlet condition on $\Gamma$ our CGO has vanishing trace on $\Gamma$. Note that by switching the sign the technique we have developed applies to $\Gamma_\pm$ compactly contained in $\{z\in \partial\Omega \mid \pm g(\partial_{y_1},\nu(z)) >0\}$ if we consider ansatz in \eqref{ansatz} of the form $e^{\pm(\phi+i\psi)/h}  (a + a^\pm_h) $. Therefore, we are able to construct CGOs $u_\pm$ vanishing on $\Gamma_\pm$ respectively:
\begin{prop}
\label{CGO}
Let $\Omega$ be a bounded smooth domain in a CTA manifold $(M,g)$ and $\Gamma_\pm\subset\partial\Omega$ be an open subset compactly contained in 
$$\{z\in \partial\Omega \mid \pm g(\partial_{y_1}, \nu(z) ) >0\}.$$ 
Given $q\in L^{n/2}$ one can find solutions to
\[(\Delta_g + q) u_\pm = 0,\ \ \ u_\pm\in H^1(\Omega),\ \ u_\pm\mid_{\Gamma_\pm} = 0\]
of the form
\[u_\pm = e^{\frac{\pm(\phi +i\psi)}{h}}(a + a^\pm_h + r)\]
where $a$ is as in \eqref{amplitude}, $\|a^\pm_h\|_{L^\infty} \leq C$, $a^\pm_h \to 0$ pointwise in $\Omega$ as $h\to 0$. The error term $r\in L^p$ is bounded by $\|r\|_{L^2} = o(1)$ and $\|r\|_{p} \leq C$ as $h\to 0$.
\end{prop}

\end{subsection}
\end{section}
\begin{section}{Proof of Theorem \ref{main theorem}}

We prove Theorem \ref{main theorem} using the ideas of \cite{DosKenSal}. The procedure is standard so we only give a sketch here. Let $\Gamma_\pm$ be open sets such that \[\partial\Omega\backslash{\bf B} \subset\subset \Gamma_+\subset\subset \{z\in \partial\Omega \mid g(\partial_{y_1},\nu(z) > 0)\},\ \ \partial\Omega\backslash{\bf F} \subset\subset \Gamma_-\subset\subset \{z\in \partial\Omega \mid g(\partial_{y_1}, \nu(z)) < 0\}.\]

Using Proposition \ref{CGO} we construct solutions $u_\pm \in H^1(\Omega)$ solving \[(\Delta + q_1) u_+ = 0,\ \ \ u_+\mid_{\Gamma_+} = 0,\ \ (\Delta+q_2) u_- = 0,\ \ \ u_- \mid_{\Gamma_-} = 0\] of the form
\[u_\pm = e^{\frac{\pm(\phi +i\psi)}{h}}(a_\pm + a^\pm_h + r_\pm),\ \ \ \|r_\pm\|_{L^2}= o(1),\ \ \ \|r_\pm\|_{L^p} = O(1).\]
where $a_\pm$ are of the form \eqref{amplitude}.

Since $u_\pm$ are solutions belonging to $H^1(\Omega)$ and with vanishing trace on $\partial\Omega\backslash{\bf B}$ and $\partial\Omega\backslash{\bf F}$ respectively, we have the following boundary integral identity (see Lemma A.1 of \cite{DosKenSal})
\[\int_\Omega  u_- (q_1-q_2) u_+ = 0.\]
Inserting the expressions for $u_\pm$ gives
{\small\begin{eqnarray}
\label{integrate to zero}
0=\int_\Omega q(a_+  a_- +   a^-_h a^+_h +   a^-_h a_++ a^+_h   a_-+   a^-_h r_+ + a^+_h r_- + a_+r_- +   a_-r_+ + r_+r_-)
\end{eqnarray}}
where $q= q_1 -q_2$. Writing $q$ as the sum of a $L^\infty$ function with an arbitrarily small $L^{\frac{n}{2}}$ function we see using the estimates of Proposition \ref{CGO} that in the limit as $h\to 0$ the only surviving term is

\[ 0= \int_\Omega q a_+ a_-dV_g=\int_{0}^\infty \int_{{\mathcal S}^{n-2}}  \int_{-\infty}^\infty q(y_1, t,\theta) e^{i(y_1\lambda + i\lambda t)} \beta(\theta)dy_1  d\theta dr.\]
Where the function $\beta$, the coordinates $\theta$ and $t$ are chosen as in the definition of \eqref{amplitude}. The proof now follows precisely as in \cite{DosKenSal} to show $q = 0$.\qed

\end{section}

\end{document}